\documentclass[pdflatex,sn-mathphys-num]{sn-jnl}

\usepackage{mathtools}%
\usepackage{graphicx}%
\usepackage{multirow}%
\usepackage{amsmath}
\allowdisplaybreaks
\usepackage{amsmath,amssymb,amsfonts}%
\usepackage{amsthm}%
\usepackage{mathrsfs}%
\usepackage[title]{appendix}%
\usepackage{xcolor}%
\usepackage{textcomp}%
\usepackage{manyfoot}%
\usepackage{booktabs}%
\usepackage{algorithm}%
\usepackage{algorithmicx}%
\usepackage{algpseudocode}%
\usepackage{listings}%
\usepackage{mathrsfs}%

\theoremstyle{thmstyleone}%
\newtheorem{theorem}{Theorem}[subsection]
\newtheorem{corollary}{Corollary}[subsection]
\newtheorem{lemma}{Lemma}[subsection]
\newtheorem{remark}{Remark}[subsection]

\newtheorem{example}{Example}[subsection]
\newtheorem{definition}{Definition}[subsection]%

\theoremstyle{thmstyletwo}%

\theoremstyle{thmstylethree}%
\begin{document}
	
	\title[Article Title]{Colored Partitions and Minimal Excludants}
	%
	%
	\author[1]{\fnm{Suverchla} \sur{Arora}}\email{ssuverchla.phd23@gmail.com}
	\equalcont{These authors contributed equally to this work.}
	
	\author[1]{\fnm{Meenakshi} \sur{Rana}}\email{mrana@thapar.edu}
	\equalcont{These authors contributed equally to this work.}
	
	\author*[2]{\fnm{Shruti} \sur{Sharma}}\email{sharma\_stary2k@yahoo.co.in}
	\equalcont{These authors contributed equally to this work.}
	
	\author[1]{\fnm{Nilanjan} \sur{Bag}}\email{nilanjanb2011@gmail.com}
	\equalcont{These authors contributed equally to this work.}
	
	\affil[1]{\orgdiv{Department of Mathematics}, \orgname{Thapar Institute of Engineering and Technology}, \orgaddress{\street{Bhadson Road}, \city{Patiala}, \postcode{147004}, \state{Punjab}, \country{India}}}
	
	\affil*[2]{\orgdiv{Yadavindra Department of Sciences}, \orgname{Punjabi University Guru Kashi Campus}, \orgaddress{\street{Talwandi Sabo}, \city{Bhatinda}, \postcode{151302}, \state{Punjab}, \country{India}}}
	%
	%
	\abstract{In this article, we introduce the concept of minimal excludants in colored partitions and extend it to the framework of least $r$-gaps. In particular, we focus on $k$-colored partitions and cubic partitions, deriving the generating functions for the sums of minimal excludants and least $r$-gaps, and analyzing their asymptotic behaviors. Notably, the presented theory is comprehensively extended by the systematic generalization of each result for minimal excludants in $k$-colored partitions to the case of least $r$-gaps. We introduce generalized cubic partitions and discover the minimal excludant, as well as several additional results.}
	
	
	
	\maketitle
	
	\section{Introduction}\label{sec1}
	A partition $\Lambda$ of a positive integer $\nu$ is a finite non-increasing sequence of positive integers $(\Lambda_{1}, \Lambda_{2},\ldots, \Lambda_{p})$ such that 
	$$\Lambda\vdash\Lambda_{1}+\Lambda_{2}+\cdots+\Lambda_{p}=\nu,$$
	where the $\Lambda_{j}$'s, for $1\leq j\leq p$, are called the parts of the partition $\Lambda.$ Let the set of all partitions of $\nu$ be denoted by $\mathcal{P}(\nu),$ and let the count of partitions of $\nu$ be denoted by $|\mathcal{P}(\nu)|.$ For example, $\mathcal{P}(5)=\left\{5,~4+1,~3+2,~3+1+1,~2+2+1,~2+1+1+1,~1+1+1+1+1\right\}$. Thus, the count of partitions of $5,$ denoted by $|\mathcal{P}(5)|$ is $7.$
	The generating function for $|\mathcal{P}(\nu)|$ is
	\begin{eqnarray*}
		\sum\limits_{\nu=0}^{\infty}|\mathcal{P}(\nu)|q^{\nu}=\prod_{\nu=1}^{\infty}\frac{1}{(1-q^{\nu})}=\frac{1}{(q;q)_\infty},
	\end{eqnarray*}
	where, $|q| < 1$ and
	\begin{eqnarray*}
		(a;q)_\infty:=\prod_{\nu=1}^{\infty}(1-aq^{\nu-1}).
	\end{eqnarray*}
	Throughout this article, we also fix
	\begin{eqnarray*}
		(a;q)_0&:=1 ~~~\text{and}~~
		(a;q)_t&:=\prod_{\nu=1}^{t}(1-aq^{\nu-1})~~~\text{with}~ t\in\mathbb{N}.
	\end{eqnarray*}
	In 2006, Grabner and Knopfmacher \cite{grabner} introduced a significant concept in the theory of partitions, known as the smallest gap. The smallest gap of a partition $\Lambda$ is the least positive integer that is not a part of $\Lambda.$ In 2019, Andrews and Newman \cite{andrews_1} revisited the definition of the smallest gap of a partition $\Lambda$ and renamed as the minimal excludant, denoted by $mex(\Lambda).$ The sum of the minimal excludants across all the partitions of $\nu$ is denoted by $\sigma mex(\nu)$ with the corresponding generating function
	\begin{equation}\label{mex}
		\sum\limits_{\nu=0}^{\infty}\sigma mex(\nu) q^{\nu} = (-q;q)^2_{\infty},
	\end{equation}
	where
	$
	\sigma mex(\nu) := \sum_{\Lambda \in \mathcal{P}(\nu)}mex(\Lambda).
	$
	Grabner and Knopfmacher \cite{grabner} gave an asymptotic formula for $\sigma{mex(\nu)},$ which is
	\begin{eqnarray}\label{asymp_mex}
		\sigma{mex(\nu)}\sim \frac{1}{4\sqrt[4]{6\nu^3}}\left(\exp\!\left( \pi \sqrt{\frac{2\nu}{3}} \right)\right)~~~~
		\text{as}~~ \nu \to \infty.
	\end{eqnarray}
	Also, Andrews and Newman \cite{andrews_1} explored the concept of the minimal odd excludant of a partition $\Lambda,$ denoted by $moex(\Lambda),$ defined as the smallest positive odd integer that does not appear as a part of $\Lambda.$ The sum of the minimal odd excludants across all the partitions of $\nu$ is denoted by $\sigma moex(\nu)$ with the corresponding generating function
	\begin{eqnarray}
		\sum\limits_{\nu=0}^{\infty}\sigma moex(\nu) q^{\nu} = (-q;q^2)^2_{\infty}(-q;q)_{\infty},\label{moex}
	\end{eqnarray}
	where
	$
	\sigma moex(\nu) :=  \sum_{\Lambda \in \mathcal{P}(\nu)}moex(\Lambda). 
	$
	In  2024, Kaur \cite{KaurT} introduced the concept of the minimal even excludant of a partition $\Lambda,$ denoted by $meex(\Lambda),$ defined as the smallest positive even integer that does not appear as a part of $\Lambda.$ The sum of the minimal even excludants across all the partitions of $\nu$ is denoted by $\sigma meex(\nu)$ with the corresponding generating function
	\begin{eqnarray}
		\sum\limits_{\nu=0}^{\infty}\sigma meex(\nu) q^{\nu} =2(-q^2;q^2)^2_{\infty}(-q;q)_{\infty}, \label{meex}
	\end{eqnarray}
	where
	$
	\sigma meex(\nu) := \sum_{\Lambda \in \mathcal{P}(\nu)}meex(\Lambda).
	$
	In 2024, Baruah \textit{et al.}
	\cite{baruah} extended the study of minimal excludants by introducing two new functions $\sigma_o mex(\nu)$ and $\sigma_e mex(\nu)$. For a positive integer $\nu,$
	\begin{eqnarray*}
		\sigma_o mex(\nu):=\sum\limits_{\substack{\Lambda\in \mathcal{P}(\nu)\\ 2\nmid mex(\Lambda)}}mex(\Lambda) ~~~\text{and }~~~ \sigma_e mex(\nu):=\sum\limits_{\substack{\Lambda\in \mathcal{P}(\nu)\\ 2\mid mex(\Lambda)}}mex(\Lambda).
	\end{eqnarray*}
	Moreover, they established the following generating functions:
	\begin{eqnarray}
		\sum\limits_{\nu=0}^{\infty}\sigma_o mex(\nu)q^{\nu}&=&\frac{1}{2}\left((-q;q)^2_\infty+(q;q)^2_\infty\right) \label{sigma_o},\\
		\sum\limits_{\nu=0}^{\infty}\sigma_e mex(\nu)q^{\nu}&=&\frac{1}{2}\left((-q;q)^2_\infty-(q;q)^2_\infty\right).\label{sigma_e}
	\end{eqnarray}
	In 2010, for a positive integer $r$, Wagner \cite{Wagner} studied the smallest gap as the smallest integer whose multiplicity is less than $r$. In particular, for $r=1,$ one can obtain the smallest gap.
	In 2020, Ballantine and Merca \cite{Ballantine_1} revisited the definition of the smallest gap of a partition $\Lambda,$ referred to it as the least $r$-gap, and denoted it by $g_r(\Lambda)$. The sum of the least $r$-gaps across all the partitions of $\nu$ is denoted by $\sigma_r{mex}(\nu),$ with the corresponding generating function
	\begin{equation}\label{least_gap_merca}
		\sum\limits_{\nu=0}^{\infty}\sigma_rmex(\nu)q^{\nu}=\frac{ (q^{2r};q^{2r})_\infty}{(q^r;q^{2r})_\infty(q;q)_\infty},
	\end{equation}
	where
	$\sigma_r{mex}(\nu):=\sum_{\Lambda \in \mathcal{P}(\nu)} g_r(\Lambda).$\\
	
	In this paper, we will extend these studies on minimal excludants to colored partitions. 
	In Section \ref{sec2}, we have discussed all the necessary preliminaries; furthermore, we have provided definitions of minimal excludants and least $r$-gaps in the context of colored partitions. In Section \ref{sec3}, we obtain generating functions for the sum of minimal excludants over $k$-colored partitions; in addition, we examine their asymptotic behavior. Section \ref{sec4} is devoted to the study of least $r$-gaps in the context of $k$-colored partitions, wherein we provide generating functions for these gaps and analyze their asymptotic behavior. Section \ref{sec5} addresses minimal excludants over generalized cubic partitions.
	\section{Preliminaries}\label{sec2}
	In the literature, colored partitions appear in many different ways like $n$-colored partitions \cite{n_color}, $k$-colored partitions \cite{k_color}, cubic partitions \cite{Merca_cubic} etc. We define a general version of the colored partitions which includes all these types so that our definition of minimal excludant suffices for each type of colored partition. A colored partition $\Lambda$ of a positive integer $\nu$ is a partition in which each part may occur in multiple colors. We call these colored parts as colored positive integers. For denoting different colors of a part, we use subscripts. If each part can appear with a fixed number of subscripts, say $k,$ then these are called $k$-colored partitions and if each part can occur with as many subscripts as the size of the part then these are called $n$-colored partitions. Here, for any part $\alpha_i,$ $\alpha$ will be called the size and $i$ is the color of the part. Thus, for 2-colored partitions a part of any size can occur with subscripts 1 or 2. For $n$-colored partitions, a part of size $n$ can occur as $n_1,n_2,n_3,\cdots,n_n$. The parts of colored partitions follow lexicographic ordering; i.e., for any two parts $\alpha_i$ and $\beta_j,$
	$$
	\alpha_i 
	\begin{cases}
		< \beta_j, & \text{if } (\alpha=\beta  ~\text{and}~ i< j)~\text{or}~(\alpha<\beta) ;\\
		> \beta_j, & \text{if } (\alpha=\beta  ~\text{and}~ i> j)~\text{or}~(\alpha>\beta) ;\\
		=\beta_j, & \text{otherwise}.
	\end{cases}
	$$
	Keeping this order structure in mind, we can now write any type of colored partitions of any positive integer $\nu$ using the definition of partitions. For example, there are ten 2-colored partitions of $3,$ given by $3_1,3_2,2_1+1_1,2_1+1_2,2_2+1_1,2_2+1_2,1_1+1_1+1_1,1_2+1_1+1_1,1_2+1_2+1_1,1_2+1_2+1_2.$\\ 
	With this terminology, the minimal excludant over any type of colored partitions can be defined as ``the smallest colored positive integer that is not a part of the partition". To define the minimal odd (even) excludant, we will agree to consider the parity of the size of the colored part. Thus ``the minimal odd (even) excludant of a colored partition $\Lambda$ is the smallest colored positive integer with odd (even) size that does not appear as a part of $\Lambda$". Now, the minimal excludant is generalized to least $r$-gap as follows:
	\begin{definition}
		For a fixed positive integer $r,$ the least $r$-gap of a colored partition $\Lambda$ of a positive integer $\nu$ is defined as the smallest colored positive integer that does not occur at least $r$ times as a part of $\Lambda$. 
	\end{definition}
	\noindent The least odd (even) $r$-gap over colored partitions can be generalized in a similar manner by considering the parity of the part size only. \\
	
	Since we are studying minimal excludants for $k$-colored partitions, we fix certain notation for these partitions. The set of $k$-colored partitions of a positive integer $\nu$ will be denoted by $\mathcal{P}_k(\nu)$, the minimal excludant for a $k$-colored partition $\Lambda$ by $mex_k(\Lambda)$,  and the sum of minimal excludants for all the $k$-colored partitions of $\nu$ by $\sigma mex_k(\nu)$; that is,
	\begin{align*}
		\sigma mex_k(\nu) &:= \sum\limits_{\Lambda \in \mathcal{P}_k(\nu)}mex_k(\Lambda).
	\end{align*}
	The notation for minimal odd (even) excludant is also extended to $k$-colored partitions by simply adding subscript $k$ to the notations for ordinary partitions. The notations $\sigma_o mex_k(\nu)$ and $\sigma_e mex_k(\nu)$ also follow similarly.\\
	To denote least $r$-gap, least odd $r$-gap and least even $r$-gap of a $k$-colored partition $\Lambda$ we use $gap_k^r(\Lambda),$ $ogap_k^r(\Lambda),$ and $egap_k^r(\Lambda),$ respectively. The notations for the sums of these $r$-gaps are obtained by simply appending $\sigma$ to the notations of $r$-gaps. Thus, 
	\begin{align*}
		\sigma gap_k^r(\nu) \!:=\!\!\!\!\!\!\sum\limits_{\Lambda \in \mathcal{P}_k(\nu)}\!\!\!\!gap_k^r(\Lambda), ~
		\sigma ogap_k^r(\nu) \!:=\!\!\!\!\!\!\sum\limits_{\Lambda \in \mathcal{P}_k(\nu)}\!\!\!\!ogap_k^r(\Lambda), ~\text{and}~
		\sigma egap_k^r(\nu) \!:=\!\!\!\!\!\!\sum\limits_{\Lambda \in \mathcal{P}_k(\nu)}\!\!\!\!egap_k^r(\Lambda).
	\end{align*}
	Also,
	\begin{align*}
		\sigma_o gap_k^r(\nu):=\sum\limits_{\substack{\Lambda\in \mathcal{P}_k(\nu)\\ 2\nmid gap_k^r(\Lambda)}}gap_k^r(\Lambda),~~~
		\sigma_e gap_k^r(\nu):=\sum\limits_{\substack{\Lambda\in \mathcal{P}_k(\nu)\\ 2\mid gap_k^r(\Lambda)}}gap_k^r(\Lambda).
	\end{align*}
	Next, in order to have a clear understanding of the minimal excludants over $k$-colored partitions, we take up an example.
	\begin{example}
		The minimal excludants and their sums corresponding to all $2$-colored partitions of $3$ are as follows:
	\end{example}
	\begin{center}
		\scalebox{0.9}{\begin{tabular}{|c|c|c|c|}
				\hline
				2-colored partitions $\Lambda$ of 3 & $ mex_{2}(\Lambda)$&$ moex_{2}(\Lambda)$&$ meex_{2}(\Lambda)$\\
				\hline
				$3_1$ & $1_1$&$1_1$&$2_1$\\
				$3_2$ & $1_1$&$1_1$&$2_1$\\
				$2_1+1_1$ & $1_2$&$1_2$&$2_2$\\
				$2_1+1_2$ & $1_1$&$1_1$&$2_2$\\
				$2_2+1_1$ & $1_2$&$1_2$&$2_1$\\
				$2_2+1_2$ & $1_1$&$1_1$&$2_1$\\
				$1_1+1_1+1_1$ & $1_2$&$1_2$&$2_1$\\
				$1_2+1_1+1_1$ & $2_1$&$3_1$&$2_1$\\
				$1_2+1_2+1_1$ & $2_1$&$3_1$&$2_1$\\
				$1_2+1_2+1_2$ & $1_1$&$1_1$&$2_1$\\
				\hline
				& $\sigma mex_{2}(3)$=12&$\sigma moex_{2}(3)$=14&$\sigma meex_{2}(3)$=20\\
				\hline
		\end{tabular}}
	\end{center}
	Using standard techniques of partition theory, we can write the generating function for $|\mathcal{P}_k(\nu)|$(the count of $k$-colored partitions of $\nu$) as 
	\begin{eqnarray*}
		\sum\limits_{\nu=0}^{\infty}|\mathcal{P}_k(\nu)|q^{\nu}=\prod_{\nu=1}^{\infty}\frac{1}{(1-q^{\nu})^k}=\frac{1}{(q;q)_\infty^k}.
	\end{eqnarray*}
	To express some of the generating functions, for minimal excludants in closed form, we will make use of the Jacobi's Triple Product Identity (JTP) \cite[eq.(1.1.1),~p.~1]{power_of_q}. For $a\neq0$ and $|q|<1,$ we have
	\begin{eqnarray}\label{JTP}
		\sum\limits_{\nu=-\infty}^{\infty}a^\nu q^{\nu^2}=\prod_{\nu=1}^{\infty}(1+aq^{2\nu-1})(1-q^{2\nu})(1+a^{-1}q^{2\nu-1}),\nonumber
	\end{eqnarray}
	setting  $a=zq^k$ and $q=q^k$, we get
	\begin{eqnarray}\label{JTP_1}
		\sum\limits_{\nu=-\infty}^{\infty}z^\nu q^{k\nu(\nu+1)}=(-zq^{2k};q^{2k})_\infty\left(\frac{-1}{z};q^{2k}\right)_\infty(q^{2k};q^{2k})_\infty,
	\end{eqnarray}
	while $a=1$ and $q=q^k$ yield
	\begin{eqnarray}\label{JTP_2}
		\sum\limits_{\nu=-\infty}^{\infty}q^{k\nu^2}&=(-q^k;q^{2k})_\infty(-q^k;q^{2k})_\infty(q^{2k};q^{2k})_\infty=(-q^k;q^{2k})^2_\infty(q^{2k};q^{2k})_\infty.
	\end{eqnarray}
	Also, with $a=q^{l}$ and $q=q^{s+l}$, we get 
	\begin{eqnarray}\label{JTP_3}
		\sum\limits_{\nu=-\infty}^{\infty}q^{s\nu^2+l\nu(\nu+1)}&=(-q^{s+2l};q^{2(s+l)})_\infty(q^{2(s+l)};q^{2(s+l)})_\infty(-q^s;q^{2(s+l)})_\infty.
	\end{eqnarray}
	
	In obtaining the asymptotic formulas for the sum of minimal excludants, the following two identities \cite{p_s_kaur} play a significant role:
	\begin{eqnarray}
		\frac{1}{(e^{-y};e^{-y})_\infty}&\sim&\sqrt{\frac{y}{2\pi}}\exp{\left(\frac{\pi^2}{6y}\right)}\text{~~~~as~} y\rightarrow0^{+},\label{result_1}\\
		(-e^{-y};e^{-y})_\infty&\sim&\frac{1}{\sqrt2}\exp{\left(\frac{\pi^2}{12y}\right)}\text{~~~~as~} y\rightarrow0^{+}.\label{result_2}
	\end{eqnarray}
	Furthermore, we state the result of Ingham \cite{ingham}, which is integral to understanding the asymptotic behavior of the sum of minimal excludants.
	\begin{theorem}\label{ingham} 
		Let a power series $\mathcal{T}(q) = \sum\limits_{\nu=0}^{\infty}\mathcal{R}(\nu)q^{\nu}$ have radius
		of convergence 1, where the sequence $\{ \mathcal{R}(\nu)\}$ is a weakly increasing sequence of non-negative real numbers. Let $\rho,\tau \in \mathbb{R}$ and $\xi > 0$ be constants such that
		$$\mathcal{T}(e^{-y})\sim\rho y^{\tau}exp\left(\frac{\xi}{y}\right) \text{~~~as~} y\rightarrow0^{+},$$ 
		we have
		$$\mathcal{R}(\nu)\sim \frac{\rho \xi^{\frac{2\tau+1}{4}}}{2\sqrt{\pi} \nu^\frac{2\tau+3}{4}}~exp{\left(2\sqrt{\xi \nu}\right)} \text{~~~as~} \nu\rightarrow\infty.$$
	\end{theorem}
	\section{Minimal excludants in $k$-colored partitions}\label{sec3}
	In this section, we do an elaborate study of the minimal excludants in $k$-colored partitions. 
	\subsection{Generating functions for $\sigma mex_k(\nu),$ $\sigma moex_k(\nu),$ and $\sigma meex_k(\nu)$}
	\noindent In this subsection, we first focus on deriving generating functions for $\sigma mex_k(\nu),$ $\sigma moex_k(\nu),$ and $\sigma meex_k(\nu),$ which complement the results in \eqref{mex}, \eqref{moex}, and \eqref{meex}, respectively. 
	\begin{theorem}\label{Mex_k}
		For fixed $k\in\mathbb{N}$ and $\nu \in \mathbb{Z}_{\geq0},$ we have 
		\begin{eqnarray}
			\sum\limits_{\nu=0}^{\infty}\sigma{mex_{k}}(\nu)q^{\nu} = \dfrac{(-q^{k};q^{k})_\infty(q^{2k};q^{2k})_\infty}{(q;q)_{\infty}^k}.	\label{eqn_Mex_k}
		\end{eqnarray}
	\end{theorem}
	\begin{proof}[\normalfont\textbf{Proof}]
		We fix $k\in\mathbb{N}$ and define $\mathcal{M}_k(z_1,z_2,z_3,\dots,z_k,q)$ as the series in which the coefficient of $z_\mathfrak{j}^\mathfrak{m}q^{\nu}$, denoted by $|\mathcal{P}_k(\mathfrak{m}_\mathfrak{j},\nu)|,$ represents the count of $k$-colored partitions $\Lambda$ of a positive integer $\nu$ with $mex_{k}(\Lambda)$ = $\mathfrak{m}_\mathfrak{j},$ where $1\leq \mathfrak{j}\leq k.$
		\begin{align*}
			&\mathcal{M}_k(z_1,z_2,z_3,\dots,z_k,q)  := \sum\limits_{\nu=0}^{\infty}\sum\limits_{\mathfrak{m}=1}^{\infty}\sum\limits_{\mathfrak{j}=1}^{k}|\mathcal{P}_k(\mathfrak{m}_\mathfrak{j},\nu)|z_\mathfrak{j}^\mathfrak{m}q^{\nu}\\
			&= \tiny{\!\!\sum\limits_{\mathfrak{m}=1}^{\infty}\!\!\left\{\!\!\sum\limits_{\mathfrak{j}=1}^{k}\!q^{\{\underbrace{(1+\cdots+1)}_{k~\textup{times}}+\underbrace{(2+\cdots+2)}_{k~\textup{times}}+\cdots+\underbrace{((\mathfrak{m}-1)+\cdots+(\mathfrak{m}-1))}_{k~\textup{times}}+\underbrace{(\mathfrak{m}+\cdots+\mathfrak{m})}_{(\mathfrak{j}-1)~\textup{times}}\}}\!\!\!\!\!(1-q^{\mathfrak{m}})z_\mathfrak{j}^\mathfrak{m}\!\!\right\}\!\!\!\prod\limits_{\nu=1}^{\infty}\!\!\dfrac{1}{(1-q^{\nu})^k}}\\
			&=\sum\limits_{\mathfrak{m}=1}^{\infty}\sum\limits_{\mathfrak{j}=1}^{k}q^{k(1+2+\cdots+(\mathfrak{m}-1))+\mathfrak{m}(\mathfrak{j}-1)}(1-q^{\mathfrak{m}})z_\mathfrak{j}^\mathfrak{m}\left(\prod\limits_{\nu=1}^{\infty}\dfrac{1}{(1-q^{\nu})^k}\right) \\
			&=\sum\limits_{\mathfrak{m}=1}^{\infty}\sum\limits_{\mathfrak{j}=1}^{k}\underbrace{q^{\mathfrak{m}\left(\frac{k(\mathfrak{m}-1)}{2}+(\mathfrak{j}-1)\right)}(1-q^{\mathfrak{m}})z_\mathfrak{j}^\mathfrak{m}\left(\prod\limits_{\nu=1}^{\infty}\dfrac{1}{(1-q^{\nu})^k}\right) }_{\textup{Part I}}. 
		\end{align*}
		Now,
		\begin{eqnarray}
			\sum\limits_{\nu=0}^{\infty}\sigma{mex_{k}}(\nu)q^{\nu} &:=& \sum\limits_{\mathfrak{m}=1}^{\infty}\sum\limits_{\mathfrak{j}=1}^{k}\left(\left.\dfrac{\partial}{\partial z_\mathfrak{j}}\right|_{z_\mathfrak{j}=1}\textup{Part I}\right)\nonumber\\
			&=&
			\sum\limits_{\mathfrak{m}=1}^{\infty}\sum\limits_{\mathfrak{j}=1}^{k}\left\{\left.\dfrac{\partial}{\partial z_\mathfrak{j}}\right|_{z_\mathfrak{j}=1} \left(q^{\mathfrak{m}\left(\frac{k(\mathfrak{m}-1)}{2}+(\mathfrak{j}-1)\right)}(1-q^{\mathfrak{m}})z_\mathfrak{j}^\mathfrak{m}\left(\prod\limits_{\nu=1}^{\infty}\dfrac{1}{(1-q^{\nu})^k}\right)\right)\right\}\nonumber\\
			&=&\sum\limits_{\mathfrak{m}=1}^{\infty}\sum\limits_{\mathfrak{j}=1}^{k}\mathfrak{m}q^{\mathfrak{m}\left(\frac{k(\mathfrak{m}-1)}{2}+(\mathfrak{j}-1)\right)}(1-q^{\mathfrak{m}})\left(\prod\limits_{\nu=1}^{\infty}\dfrac{1}{(1-q^{\nu})^k}\right)\nonumber \\
			&=&\sum\limits_{\mathfrak{m}=1}^{\infty}\mathfrak{m}q^{\left(\frac{\mathfrak{m}k(\mathfrak{m}-1)}{2}\right)}(1-q^{\mathfrak{m}})\left(\sum\limits_{\mathfrak{j}=1}^{k}q^{\mathfrak{m}(\mathfrak{j}-1)}\right)\left(\dfrac{1}{(q;q)^k_\infty}\right)\nonumber\\
			&=&\sum\limits_{\mathfrak{m}=1}^{\infty}\mathfrak{m}q^{\frac{k(\mathfrak{m}^2-\mathfrak{m})}{2}}(1-q^{k\mathfrak{m}})\left(\dfrac{1}{(q;q)^k_\infty}\right)\nonumber\\
			&=&\left(\sum\limits_{\mathfrak{m}=1}^{\infty}\mathfrak{m}q^{\frac{k(\mathfrak{m}^2-\mathfrak{m})}{2}}-\sum\limits_{\mathfrak{m}=1}^{\infty}\mathfrak{m}q^{\frac{k(\mathfrak{m}^2+\mathfrak{m})}{2}}\right)\left(\dfrac{1}{(q;q)^k_\infty}\right)\nonumber\\
			&=&\left(\sum\limits_{\mathfrak{m}=-\infty}^{-1}-\mathfrak{m}q^{\frac{k(\mathfrak{m}^2+\mathfrak{m})}{2}}-\sum\limits_{\mathfrak{m}=1}^{\infty}\mathfrak{m}q^{\frac{k(\mathfrak{m}^2+\mathfrak{m})}{2}}\right)\left(\dfrac{1}{(q;q)^k_\infty}\right)\nonumber\\
			&=&\sum\limits_{\mathfrak{m}=-\infty}^{\infty}\mathfrak{m}q^{\frac{k\mathfrak{m}(\mathfrak{m}+1)}{2}}\left(\dfrac{-1}{(q;q)^k_\infty}\right) \label{eqn_sigma_mex_k}. 
		\end{eqnarray}
		Now, partially differentiating both side of \eqref{JTP_1} with respect to $z,$ we obtain
		\begin{align*}
			&\sum\limits_{\nu=-\infty}^{\infty}\nu z^{\nu-1} q^{k\nu(\nu+1)}\\&=(-zq^{2k};q^{2k})_\infty\left(\frac{-1}{z};q^{2k}\right)_\infty(q^{2k};q^{2k})_\infty\left(\sum\limits_{\nu=1}^{\infty}\frac{q^{2k\nu}}{1+zq^{2k\nu}}-\sum\limits_{\nu=0}^{\infty}\frac{q^{2k\nu}}{(z+q^{2k\nu})z}\right).
		\end{align*}
		When $z=1,$ the preceding equation reduces to 
		\begin{eqnarray*}
			\sum\limits_{\nu=-\infty}^{\infty}\nu q^{k\nu(\nu+1)}&=&\left(\frac{-1}{2}\right)
			(-q^{2k};q^{2k})_\infty(-1;q^{2k})_\infty(q^{2k};q^{2k})_\infty \\
			&=&-(q^{4k};q^{4k})_\infty(-q^{2k};q^{2k})_\infty. \label{eqn_kn(n+1)}
		\end{eqnarray*}
		Replacing $q$ with $q^{1/2}$ in the previous equation, we obtain
		\begin{eqnarray}\label{eqn_kn(n+1)/2}
			\sum\limits_{\nu=-\infty}^{\infty}\nu q^{\frac{k\nu(\nu+1)}{2}}=-(q^{2k};q^{2k})_\infty(-q^{k};q^{k})_\infty.
		\end{eqnarray}
		Using \eqref{eqn_sigma_mex_k} and \eqref{eqn_kn(n+1)/2}, we get
		\begin{align*}
			\sum\limits_{\nu=0}^{\infty}\sigma{mex_{k}}(\nu)q^{\nu}=\dfrac{(-q^{k};q^{k})_\infty(q^{2k};q^{2k})_\infty}{(q;q)^k_\infty}.
		\end{align*}
		This concludes the proof.
	\end{proof}
	\noindent Next, we derive the generating function that captures the sum of minimal odd excludants across all the $k$-colored partitions of a positive integer $\nu$.
	\begin{theorem}\label{Moex_k}
		For fixed $k\in\mathbb{N}$ and $\nu \in \mathbb{Z}_{\geq0},$ we have
		$$ \sum\limits_{\nu=0}^{\infty}\sigma{moex_{k}}(\nu)q^{\nu} = \dfrac{(q^{2k};q^{2k})_\infty(-q^{k};q^{2k})_{\infty}^2}{(q;q)_{\infty}^k}.	$$
	\end{theorem}
	\begin{proof}[\normalfont\textbf{Proof}]
		We fix $k\in\mathbb{N}$ and define $\mathcal{M}^{odd}_k(z_1,z_2,z_3,\dots,z_k,q)$ as the series in which the coefficient of $z_\mathfrak{j}^{(2g-1)}q^{\nu},$ denoted by $|\mathcal{P}^{odd}_k((2g-1)_\mathfrak{j},\nu)|,$ represents the count of $k$-colored partitions $\Lambda$ of a positive integer $\nu$ with $moex_{k}(\Lambda)$=$(2g-1)_\mathfrak{j}$ for some $g\in\mathbb{N},$ where $1\leq \mathfrak{j}\leq k.$
		\begin{align*}
			&\mathcal{M}^{odd}_k(z_1,z_2,z_3,\dots,z_k,q)  := \sum\limits_{\nu=0}^{\infty}\sum\limits_{g=1}^{\infty}\sum\limits_{\mathfrak{j}=1}^{k}|\mathcal{P}^{odd}_k((2g-1)_\mathfrak{j},\nu)|z_\mathfrak{j}^{(2g-1)}q^{\nu}\\
			&= \tiny{\sum\limits_{g=1}^{\infty}\sum\limits_{\mathfrak{j}=1}^{k}q^{\{\underbrace{(1+\cdots+1)}_{k~\textup{times}}+\underbrace{(3+\cdots+3)}_{k~\textup{times}}+\cdots+\underbrace{((2g-3)+\cdots+(2g-3))}_{k~\textup{times}}+\underbrace{((2g-1)+\cdots+(2g-1))}_{(\mathfrak{j}-1)~\textup{times}}\}}}\\
		\end{align*}
		\vspace{-1.5cm}
		\begin{align*}
			&~~~~~~~~~~~~~~~~~~~~~~~~~~~~~~~~~~~~~~~~~~~~~~~~~~~~~~~~\times\left(1-q^{(2g-1)}\right)z_\mathfrak{j}^{(2g-1)}{\left(\prod\limits_{\nu=1}^{\infty}\dfrac{1}{(1-q^{\nu})^k}\right)}\\
			&=\sum\limits_{g=1}^{\infty}\sum\limits_{\mathfrak{j}=1}^{k}q^{k(1+3+\cdots+(2g-3))+(2g-1)(\mathfrak{j}-1)}\left(1-q^{(2g-1)}\right)z_\mathfrak{j}^{(2g-1)}
			\left(\prod\limits_{\nu=1}^{\infty}\dfrac{1}{(1-q^{\nu})^k}\right) \\
			&=\sum\limits_{g=1}^{\infty}\sum\limits_{\mathfrak{j}=1}^{k}\underbrace{q^{k(g-1)^2+(2g-1)(\mathfrak{j}-1)}\left(1-q^{(2g-1)}\right)z_\mathfrak{j}^{(2g-1)}\left(\prod\limits_{\nu=1}^{\infty}\dfrac{1}{(1-q^{\nu})^k}\right) }_{\textup{Part I}}. 
		\end{align*}
		Now,
		\begin{eqnarray}
			&&\sum\limits_{\nu=0}^{\infty}\sigma{moex_{k}}(\nu)q^{\nu} := \sum\limits_{g=1}^{\infty}\sum\limits_{\mathfrak{j}=1}^{k}\left(\left.\dfrac{\partial}{\partial z_\mathfrak{j}}\right|_{z_\mathfrak{j}=1}\textup{Part I}\right)\nonumber\\
			&&=
			\sum\limits_{g=1}^{\infty}\sum\limits_{\mathfrak{j}=1}^{k}\left\{\left.\dfrac{\partial}{\partial z_\mathfrak{j}}\right|_{z_\mathfrak{j}=1} \left(q^{k(g-1)^2+(2g-1)(\mathfrak{j}-1)}\left(1-q^{(2g-1)}\right)z_\mathfrak{j}^{(2g-1)}\left(\prod\limits_{\nu=1}^{\infty}\dfrac{1}{(1-q^{\nu})^k}\right)  \right)\right\}\nonumber\\
			&&=\sum\limits_{g=1}^{\infty}\sum\limits_{\mathfrak{j}=1}^{k}(2g-1) q^{k(g-1)^2+(2g-1)(\mathfrak{j}-1)}\left(1-q^{(2g-1)}\right)\left(\prod\limits_{\nu=1}^{\infty}\dfrac{1}{(1-q^{\nu})^k}\right) \nonumber\\
			&&=\sum\limits_{g=1}^{\infty}(2g-1) q^{k(g-1)^2}\left(1-q^{(2g-1)}\right)\left(\sum\limits_{\mathfrak{j}=1}^{k}q^{(2g-1)(\mathfrak{j}-1)}\right)\left(\dfrac{1}{(q;q)^k_\infty}\right) \nonumber\\
			&&=\sum\limits_{g=1}^{\infty} (2g-1) q^{k(g-1)^2}\left(1-q^{k(2g-1)}\right)\left(\dfrac{1}{(q;q)^k_\infty}\right)\nonumber\\
			&&=\left(\sum\limits_{g=1}^{\infty}(2g-1) q^{k(g-1)^2}-\sum\limits_{g=1}^{\infty}(2g-1) q^{kg^2}\right)\left(\dfrac{1}{(q;q)^k_\infty}\right)\nonumber\\
			&&=\left(\sum\limits_{g=0}^{\infty}(2g+1) q^{kg^2}-\sum\limits_{g=1}^{\infty}(2g-1) q^{kg^2}\right)\left(\dfrac{1}{(q;q)^k_\infty}\right)\nonumber\\
			&&=\left(1+\sum\limits_{g=1}^{\infty}(2g+1) q^{kg^2}-\sum\limits_{g=1}^{\infty}(2g-1) q^{kg^2}\right)\left(\dfrac{1}{(q;q)^k_\infty}\right)\nonumber\\
			&&=\left(1+2\sum\limits_{g=1}^{\infty} q^{kg^2}\right)\left(\dfrac{1}{(q;q)^k_\infty}\right)
			=\left(1+\sum\limits_{g=-\infty}^{-1} q^{kg^2}+\sum\limits_{g=1}^{\infty} q^{kg^2}\right)\left(\dfrac{1}{(q;q)^k_\infty}\right)\nonumber\\
			&&=\left(\sum\limits_{g=-\infty}^{\infty} q^{kg^2}\right)\left(\dfrac{1}{(q;q)^k_\infty}\right).\label{sigma_moex_k}
		\end{eqnarray}
		Using \eqref{JTP_2} and \eqref{sigma_moex_k}, we obtain
		\begin{align*}
			\sum\limits_{\nu=0}^{\infty}\sigma{moex_{k}}(\nu)q^{\nu}&=\dfrac{(q^{2k};q^{2k})_\infty(-q^k;q^{2k})^2_\infty}{(q;q)^k_\infty}.
		\end{align*}
		This concludes the proof.
	\end{proof}
	\noindent The generating function that captures the sum of minimal even excludants across all $k$-colored partitions of a positive integer $\nu$ can be obtained in a similar manner as that for the sum of minimal odd excludants. Therefore, we provide only the result and omit the proof.
	\begin{theorem}\label{Meex_k}
		For fixed $k\in\mathbb{N}$ and $\nu \in \mathbb{Z}_{\geq0},$ we have
		$$ \sum\limits_{\nu=0}^{\infty}\sigma{meex_{k}}(\nu)q^{\nu} = \dfrac{2(-q^{2k};q^{2k})_\infty(q^{4k};q^{4k})_\infty}{(q;q)_{\infty}^k}.	$$
	\end{theorem}
	\subsection{Asymptotic behavior of $\sigma mex_k(\nu),$ $\sigma moex_k(\nu),$ and $\sigma meex_k(\nu)$}
	\noindent	Next, we study the asymptotic behavior of $\sigma mex_k(\nu),$ $\sigma moex_k(\nu),$ and $\sigma meex_k(\nu)$. 
	\begin{theorem}\label{Asymp_mex_k}
		For fixed $k\in\mathbb{N},$ we have
		$$\sigma{mex_{k}(\nu)}\sim \frac{k^{\frac{k-2}{4}}}{2^{\frac{k+3}{2}}\cdot6^{\frac{k}{4}}\cdot \nu^{\frac{k+2}{4}}}\exp\!\left( 2\pi \sqrt{\frac{k\nu}{6}} \right)~~~~ \text{as} ~\nu \to \infty.$$
		
	\end{theorem}
	\begin{remark}
		Note that when the value $k=1$ is substituted into Theorem $\ref{Asymp_mex_k},$ the resulting expression coincides with the result established by Grabner and Knopfmacher in \textnormal{\cite{grabner}} as given in equation \eqref{asymp_mex}.
	\end{remark}
	\noindent To establish our result, we begin by recalling the definition of a weakly increasing sequence. A sequence $\{a(\nu)\}$ is called weakly increasing if it satisfies $a_\nu\leq a_{\nu+1}$ for every $\nu\in \mathbb{N}$. We now proceed to prove a preliminary lemma.
	\begin{lemma}\label{lemma_1}
		For a fixed $k\in\mathbb{N},$ the sequences $\{\sigma mex_k(\nu)\},$ $\{\sigma moex_k(\nu)\},$ and $\{\sigma meex_k(\nu)\}$ are all weakly increasing.
	\end{lemma}
	\begin{proof}[\normalfont\textbf{Proof}]
		To prove this lemma, we fix $k\in\mathbb{N}$  and define a mapping
		$$\Gamma_1:\mathcal{P}_k(\nu)\rightarrow \mathcal{P}_k(\nu+1),$$
		and for each $\Lambda\in \mathcal{P}_k(\nu),$ we proceed as follows:
		\begin{itemize}
			\item [i.] If $mex_k(\Lambda)\neq1_\mathfrak{i}$ for some $1\leq i\leq k,$ then $\Gamma_1(\Lambda)$ is the partition of $\nu+1$ formed by adding $1_1$ as an additional part to $\Lambda.$
			\item [ii.] If $mex_k(\Lambda)=1_\mathfrak{i}$ for some $1\leq i\leq k,$
			then $\Gamma_1(\Lambda)$ is the partition of $\nu+1$ formed by increasing the largest part size of $\Lambda$ by $1$.
		\end{itemize}
		Clearly, $\Gamma_1$ is an injective map. Moreover, under the map $\Gamma_1,$ the minimal excludant remains unchanged. Hence, 
		$$\sigma mex_k(\nu)\leq\sigma mex_k(\nu+1).$$
		Therefore, $\{\sigma mex_k(\nu)\}$ is weakly increasing sequence.
		In a similar fashion, one can show that $\{\sigma moex_k(\nu)\}$ and $\{\sigma meex_k(\nu)\}$ are weakly increasing. Hence, we skip the details.
	\end{proof}
	\begin{proof}[\normalfont\textbf{Proof of Theorem \ref{Asymp_mex_k}}]
		For any $k\in \mathbb{N}$ 
		\begin{eqnarray}
			\mathcal{A}_k(q) &=&  \sum\limits_{\nu=0}^{\infty}\sigma{mex_{k}}(\nu)q^{\nu}= \dfrac{(q^{2k};q^{2k})_\infty(-q^{k};q^{k})_\infty}{(q;q)_{\infty}^k} \label{A_k(q)}.
		\end{eqnarray}
		Also, $\{\sigma{mex_{k}}(\nu)\}$ is a weakly increasing sequence.\\
		Upon replacing $q$ with $e^{-y}$ in the preceding equation \eqref{A_k(q)} and using \eqref{result_1} and \eqref{result_2}, we obtain
		\begin{align*}
			\mathcal{A}_k(e^{-y})&=\frac{(e^{-2ky};e^{-2ky})_\infty(-e^{-ky};e^{-ky})_\infty}{(e^{-y};e^{-y})_\infty^k}
			\sim\frac{\frac{1}{\sqrt2}e^{\frac{\pi^2}{12ky}}\left(\sqrt{\frac{y}{2\pi}}e^{\frac{\pi^2}{6y}}\right)^k}{\sqrt{\frac{ky}{\pi}}e^{\frac{\pi^2}{12ky}}}\\
			&=\frac{y^{\frac{k-1}{2}}e^{\frac{k\pi^2}{6y}}}{2^{\frac{k+1}{2}}\sqrt{k}~\pi^{\frac{k-1}{2}}} \text{~~~as~} y\rightarrow0^{+}.
		\end{align*}
		Using Theorem \ref{ingham} with $\rho=\frac{1}{2^{\frac{k+1}{2}}\sqrt{k}~\pi^{\frac{k-1}{2}}},$ $\tau=\frac{k-1}{2},$ and $\xi=\frac{k\pi^2}{6},$  we get 
		\begin{align*}
			\sigma{mex_{k}}(\nu)&\sim \frac{\frac{1}{2^{\frac{k+1}{2}}\sqrt{k}~\pi^{\frac{k-1}{2}}}\left(\frac{k\pi^2}{6}\right)^{\frac{k}{4}}}{2\sqrt{\pi} \nu^\frac{k+2}{4}}\exp\!\left(2\sqrt{\frac{k\pi^2\nu}{6}}\right)
			= \frac{k^{\frac{k-2}{4}}}{2^{\frac{k+3}{2}}\cdot6^{\frac{k}{4}}\cdot \nu^{\frac{k+2}{4}}}\exp\!\left( 2\pi \sqrt{\frac{k\nu}{6}} \right).
		\end{align*}
		This concludes the proof.
	\end{proof}
	\begin{theorem}\label{Asymp_moex_k_meex_k}
		For fixed $k\in\mathbb{N},$ we have
		$$\sigma{meex_{k}(\nu)}\sim\sigma{moex_{k}(\nu)}\sim \frac{k^{\frac{k-2}{4}}}{2^{\frac{k+2}{2}}\cdot6^{\frac{k}{4}}\cdot \nu^{\frac{k+2}{4}}}\exp\!\left( 2\pi \sqrt{\frac{k\nu}{6}} \right)~~~~\text{as}~\nu \to \infty.$$
	\end{theorem}
	\noindent We get the asymptotic behavior of $\sigma{meex(\nu)}$ and $\sigma{moex(\nu)}$  as a corollary to Theorem \ref{Asymp_moex_k_meex_k}.
	\begin{corollary}\label{asymp_meex}
		$$\sigma{meex(\nu)}\sim\sigma{moex(\nu)}\sim \frac{1}{2^{\frac{3}{2}}\cdot\sqrt[4]{6\nu^3}}\exp\!\left( \pi \sqrt{\frac{2\nu}{3}} \right)~~~~\text{as}~\nu \to \infty.$$
	\end{corollary}
	\begin{proof}[\normalfont\textbf{Proof of Theorem \ref{Asymp_moex_k_meex_k}}]
		For any $k\in \mathbb{N}$
		\begin{eqnarray}
			\mathcal{B}_k(q) :=  \sum\limits_{\nu=0}^{\infty}\sigma{moex_{k}}(\nu)q^{\nu}= \dfrac{(-q^{k};q^{2k})_{\infty}^2(q^{2k};q^{2k})_\infty}{(q;q)_{\infty}^k}=\dfrac{(-q^{k};q^{k})_{\infty}^2(q^{2k};q^{2k})_\infty}{(-q^{2k};q^{2k})_\infty^2(q;q)_{\infty}^k}.\label{B_k(q)}
		\end{eqnarray}  
		Also, $\{\sigma{moex_{k}}(\nu)\}$ is a weakly increasing sequence.\\
		Upon replacing $q$ with $e^{-y}$ in the preceding equation \eqref{B_k(q)} and using \eqref{result_1} and \eqref{result_2}, we obtain
		\begin{align*}
			\mathcal{B}_k(e^{-y})&=\frac{(-e^{-ky};e^{-ky})_\infty^2(e^{-2ky};e^{-2ky})_\infty}{(-e^{-2ky};e^{-2ky})_\infty^2(e^{-y};e^{-y})_\infty^k} \sim\frac{\left(\frac{1}{\sqrt2}e^{\frac{\pi^2}{12ky}}\right)^2\left(\sqrt{\frac{y}{2\pi}}e^{\frac{\pi^2}{6y}}\right)^k}{\left(\frac{1}{\sqrt2}e^{\frac{\pi^2}{24ky}}\right)^2\left(\sqrt{\frac{ky}{\pi}}e^{\frac{\pi^2}{12ky}}\right)}\\
			&=\frac{y^{\frac{k-1}{2}}e^{\frac{k\pi^2}{6y}}}{2^{\frac{k}{2}}\sqrt{k}~\pi^{\frac{k-1}{2}}} \text{~~~as~} y\rightarrow0^{+}.
		\end{align*}
		Using Theorem \ref{ingham} with $\rho=\frac{1}{2^{\frac{k}{2}}\sqrt{k}~\pi^{\frac{k-1}{2}}},$ $\tau=\frac{k-1}{2},$ and $\xi=\frac{k\pi^2}{6}$, we get 
		\begin{align*}
			\sigma{moex_{k}}(\nu)&\sim \frac{\frac{1}{2^{\frac{k}{2}}\sqrt{k}~\pi^{\frac{k-1}{2}}}\left(\frac{k\pi^2}{6}\right)^{\frac{k}{4}}}{2\sqrt{\pi} \nu^\frac{k+2}{4}}\exp\!\left(2\sqrt{\frac{k\pi^2\nu}{6}}\right) =\frac{k^{\frac{k-2}{4}}}{2^{\frac{k+2}{2}}\cdot6^{\frac{k}{4}}\cdot \nu^{\frac{k+2}{4}}}\exp\!\left( 2\pi \sqrt{\frac{k\nu}{6}} \right).
		\end{align*}
		Similarly, for any $k\in\mathbb{N}$
		\begin{eqnarray}
			\mathcal{C}_k(q) :=  \sum\limits_{\nu=0}^{\infty}\sigma{meex_{k}}(\nu)q^{\nu}= \dfrac{2(q^{4k};q^{4k})_\infty(-q^{2k};q^{2k})_\infty}{(q;q)_{\infty}^k}. \label{C_k(q)}
		\end{eqnarray}
		Also, $\{\sigma{meex_{k}}(\nu)\}$ is a weakly increasing sequence.\\ 
		Upon replacing $q$ with $e^{-y}$ in the preceding equation \eqref{C_k(q)} and using \eqref{result_1} and \eqref{result_2}, we obtain
		\begin{align*}
			\mathcal{C}_k(e^{-y})&=\frac{2(e^{-4ky};e^{-4ky})_\infty(-e^{-2ky};e^{-2ky})_\infty}{(e^{-y};e^{-y})_\infty^k}
			\sim\frac{\frac{2}{\sqrt2}e^{\frac{\pi^2}{24ky}}\left(\sqrt{\frac{y}{2\pi}}e^{\frac{\pi^2}{6y}}\right)^k}{\sqrt{\frac{4ky}{2\pi}}e^{\frac{\pi^2}{24ky}}}\\
			&=\frac{y^{\frac{k-1}{2}}e^{\frac{k\pi^2}{6y}}}{2^{\frac{k}{2}}\sqrt{k}~\pi^{\frac{k-1}{2}}}, \text{~~~as~} y\rightarrow0^{+}.
		\end{align*}
		Using Theorem \ref{ingham} with $\rho=\frac{1}{2^{\frac{k}{2}}\sqrt{k}~\pi^{\frac{k-1}{2}}},$ $\tau=\frac{k-1}{2},$ and $\xi=\frac{k\pi^2}{6}$, we get 
		\begin{align*}
			\sigma{meex_{k}}(\nu)&\sim \frac{\frac{1}{2^{\frac{k}{2}}\sqrt{k}~\pi^{\frac{k-1}{2}}}\left(\frac{k\pi^2}{6}\right)^{\frac{k}{4}}}{2\sqrt{\pi} \nu^\frac{k+2}{4}}\exp\!\left(2\sqrt{\frac{k\pi^2\nu}{6}}\right)
			=\frac{k^{\frac{k-2}{4}}}{2^{\frac{k+2}{2}}\cdot6^{\frac{k}{4}}\cdot \nu^{\frac{k+2}{4}}}\exp\!\left( 2\pi \sqrt{\frac{k\nu}{6}} \right).
		\end{align*}
		This concludes the proof.
	\end{proof}
	\subsection{Analogous results for $\sigma_o mex(\nu)$ and $\sigma_e mex(\nu)$ in $k$-colored partitions}
	\noindent In this subsection, we derive the generating functions for $\sigma_o mex_k(\nu)$ and $\sigma_e mex_k(\nu),$ which complement the results given by \eqref{sigma_o} and \eqref{sigma_e}.
	\begin{theorem} \label{baruah_odd}
		For fixed $k\in\mathbb{N}$ and $\nu \in \mathbb{Z}_{\geq0},$ we have
		\begin{eqnarray}\label{eqn_baruah_odd}
			\sum\limits_{\nu=0}^{\infty}\sigma_o mex_k(\nu)q^{\nu} =\dfrac{1}{2(q;q)^k_\infty}\left\{(-q^k;q^k)_\infty(q^{2k};q^{2k})_\infty+(q^k;q^k)^3_\infty\right\}.
		\end{eqnarray}
	\end{theorem}
	\begin{proof}[\normalfont\textbf{Proof}]
		We fix $k\in\mathbb{N}$ and define $\mathcal{M}^{O}_k(z_1,z_2,z_3,\dots,z_k,q)$ as the series in which the coefficient of $z_\mathfrak{j}^\mathfrak{m}q^{\nu}$, denoted by $|\mathcal{P}^{O}_k((2g+1)_\mathfrak{j},\nu)|,$ represents the count of $k$-colored partitions $\Lambda$ of a positive integer $\nu$ with $mex_{k}(\Lambda)$=$(2g+1)_\mathfrak{j}$ for some $g\in\mathbb{Z}_{\geq0},$ where $1\leq \mathfrak{j} \leq k$.
		\begin{eqnarray*}
			&&\mathcal{M}^{O}_k(z_1,z_2,z_3,\dots,z_k,q)  := \sum\limits_{\nu=0}^{\infty}\sum\limits_{g=0}^{\infty}\sum\limits_{\mathfrak{j}=1}^{k}|\mathcal{P}^{O}_k((2g+1)_\mathfrak{j},\nu)|z_\mathfrak{j}^{(2g+1)}q^{\nu}\\
			&&= \tiny{\sum\limits_{g=0}^{\infty}\sum\limits_{\mathfrak{j}=1}^{k}q^{\{\underbrace{(1+\cdots+1)}_{k~\textup{times}}+\underbrace{(2+\cdots+2)}_{k~\textup{times}}+\cdots+\underbrace{(2g+\cdots+2g)}_{k~\textup{times}}+\underbrace{((2g+1)+\cdots+(2g+1))}_{(\mathfrak{j}-1)~\textup{times}}\}}}~~~~~~~~~\\[-3mm]
			&&~~~~~~~~~~~~~~~~~~~~~~~~~~~~~~~~~~~~~~~~~~~~~~~~~~\times\left(1-q^{(2g+1)}\right)z_\mathfrak{j}^{(2g+1)}\left(\prod\limits_{\nu=1}^{\infty}\dfrac{1}{(1-q^{\nu})^k}\right)\\
			&&=\sum\limits_{g=0}^{\infty}\sum\limits_{\mathfrak{j}=1}^{k}q^{\left\{k(1+2+\cdots+2g)+(2g+1)(\mathfrak{j}-1)\right\}}\left(1-q^{(2g+1)}\right)z_\mathfrak{j}^{(2g+1)}\left(\prod\limits_{\nu=1}^{\infty}\dfrac{1}{(1-q^{\nu})^k}\right) \\
			&&=\sum\limits_{g=0}^{\infty}\sum\limits_{\mathfrak{j}=1}^{k}\underbrace{q^{\left\{\frac{2kg(2g+1)}{2}+(2g+1)(\mathfrak{j}-1)\right\}}\left(1-q^{(2g+1)}\right)z_\mathfrak{j}^{(2g+1)}\left(\prod\limits_{\nu=1}^{\infty}\dfrac{1}{(1-q^{\nu})^k}\right) }_{\textup{Part I}}. 
		\end{eqnarray*}
		Now,
		\begin{align*}
			&\sum\limits_{\nu=0}^{\infty}\sigma_o{mex_{k}}(\nu)q^{\nu} := \sum\limits_{g=0}^{\infty}\sum\limits_{\mathfrak{j}=1}^{k}\left(\left.\dfrac{\partial}{\partial z_\mathfrak{j}}\right|_{z_\mathfrak{j}=1}\textup{Part I}\right)\\
			&=
			\sum\limits_{g=0}^{\infty}\sum\limits_{\mathfrak{j}=1}^{k}\left\{\left.\dfrac{\partial}{\partial z_\mathfrak{j}}\right|_{z_\mathfrak{j}=1} \left(q^{\left\{\frac{2kg(2g+1)}{2}+(2g+1)(\mathfrak{j}-1)\right\}}\left(1-q^{(2g+1)}\right)z_\mathfrak{j}^{(2g+1)}\left(\prod\limits_{\nu=1}^{\infty}\dfrac{1}{(1-q^{\nu})^k}\right)  \right)\right\}\\
			&=\sum\limits_{g=0}^{\infty}\sum\limits_{\mathfrak{j}=1}^{k}(2g+1) q^{\left\{\frac{2kg(2g+1)}{2}+(2g+1)(\mathfrak{j}-1)\right\}}(1-q^{(2g+1)})\left(\prod\limits_{\nu=1}^{\infty}\dfrac{1}{(1-q^{\nu})^k}\right) \\
			&=\sum\limits_{g=0}^{\infty}(2g+1) q^{\frac{2kg(2g+1)}{2}}\left(1-q^{(2g+1)}\right)\left(\sum\limits_{\mathfrak{j}=1}^{k}q^{(2g+1)(\mathfrak{j}-1)}\right) \left(\dfrac{1}{(q;q)^k_\infty}\right)\\
			&=\sum\limits_{g=0}^{\infty} (2g+1) q^{\frac{2kg(2g+1)}{2}}\left(1-q^{k(2g+1)}\right)\left(\dfrac{1}{(q;q)^k_\infty}\right)\\
			&=\left(\sum\limits_{g=0}^{\infty}(2g+1)q^{\frac{k\cdot2g(2g+1)}{2}}-\sum\limits_{g=0}^{\infty}(2g+1) q^{\frac{k\cdot(2g+2)(2g+1)}{2}}\right)\left(\dfrac{1}{(q;q)^k_\infty}\right)\\
			&=\left(\sum\limits_{g=0}^{\infty}(2g+1)q^{k\binom{2g+1}{2}}-\sum\limits_{g=0}^{\infty}(2g+1) q^{k\binom{2g+2}{2}}-\sum\limits_{g=0}^{\infty}q^{k\binom{2g+2}{2}}+\sum\limits_{g=0}^{\infty}q^{k\binom{2g+2}{2}}\right)\left(\dfrac{1}{(q;q)^k_\infty}\right)\\
			&=\left(\sum\limits_{g=0}^{\infty}(2g+1)q^{k\binom{2g+1}{2}}-\sum\limits_{g=0}^{\infty}(2g+2) q^{k\binom{2g+2}{2}}+\sum\limits_{g=0}^{\infty}q^{k\binom{2g+2}{2}}\right)\left(\dfrac{1}{(q;q)^k_\infty}\right)\\\
			&=\left(\sum\limits_{g=0}^{\infty}(g+1)(-1)^gq^{k\binom{g+1}{2}}+\sum\limits_{g=0}^{\infty}q^{k\binom{2g+2}{2}}\right)\left(\dfrac{1}{(q;q)^k_\infty}\right).
		\end{align*}
		By multiplying and dividing the right hand side by 2 and then using Jacobi's identity  \cite[Eq.(10.4.9),p.~500]{special_functions}, Gauss's identity \cite[Eq.(2.2.13),p.~23]{andrews_b} and Euler's identity \cite[Eq.(1.2.5),p.~5]{andrews_b}, we obtain
		\begin{align*}
			&\sum\limits_{\nu=0}^{\infty}\sigma_o{mex_{k}}(\nu)q^{\nu}=\left(\sum\limits_{g=0}^{\infty}(2g+2)(-1)^gq^{k\binom{g+1}{2}}+\sum\limits_{g=0}^{\infty}2q^{k\binom{2g+2}{2}}\right)\left(\dfrac{1}{2(q;q)^k_\infty}\right)\\
			&=\left(\sum\limits_{g=0}^{\infty}(2g+1)(-1)^gq^{k\binom{g+1}{2}}+\sum\limits_{g=0}^{\infty}(-1)^gq^{k\binom{g+1}{2}}+\sum\limits_{g=0}^{\infty}2q^{k\binom{2g+2}{2}}\right)\left(\dfrac{1}{2(q;q)^k_\infty}\right)\\
			&=\left(\sum\limits_{g=0}^{\infty}(2g+1)(-1)^gq^{k\binom{g+1}{2}}+\sum\limits_{g=0}^{\infty}q^{k\binom{2g+1}{2}}-\sum\limits_{g=0}^{\infty}q^{k\binom{2g+2}{2}}+\sum\limits_{g=0}^{\infty}2q^{k\binom{2g+2}{2}}\right)\left(\dfrac{1}{2(q;q)^k_\infty}\right)\\
			&=\left(\sum\limits_{g=0}^{\infty}(2g+1)(-1)^gq^{k\binom{g+1}{2}}+\sum\limits_{g=0}^{\infty}q^{k\binom{2g+1}{2}}+\sum\limits_{g=0}^{\infty}q^{k\binom{2g+2}{2}}\right)\left(\dfrac{1}{2(q;q)^k_\infty}\right)\\
			&=\left(\sum\limits_{g=0}^{\infty}(2g+1)(-1)^gq^{k\binom{g+1}{2}}+\sum\limits_{g=0}^{\infty}q^{k\binom{g+1}{2}}\right)\left(\dfrac{1}{2(q;q)^k_\infty}\right)\\
			&=\left((q^k;q^k)^3_\infty+\frac{(q^{2k};q^{2k})_\infty}{(q^k;q^{2k})_\infty}\right)\left(\dfrac{1}{2(q;q)^k_\infty}\right)\\
			&=\dfrac{1}{2(q;q)^k_\infty}\left\{(-q^k;q^k)_\infty(q^{2k};q^{2k})_\infty+(q^k;q^k)^3_\infty\right\}.
		\end{align*}
		This concludes the proof.
	\end{proof}
	\begin{theorem} \label{baruah_even}
		For fixed $k\in\mathbb{N}$ and $\nu \in \mathbb{Z}_{\geq0},$ we have
		\begin{eqnarray*}
			\sum\limits_{\nu=0}^{\infty}\sigma_e mex_k(\nu)q^{\nu}=\dfrac{1}{2(q;q)^k_\infty}\left\{(-q^k;q^k)_\infty(q^{2k};q^{2k})_\infty-(q^k;q^k)^3_\infty\right\}.\label{eqn_baruah_even}
		\end{eqnarray*}
	\end{theorem}
	\begin{proof}[\normalfont\textbf{Proof}]
		Subtracting equation \eqref{eqn_baruah_odd} from  equation \eqref{eqn_Mex_k} yields the result.
	\end{proof}
	
	\section{Least $r$-gaps in $k$-colored partitions}\label{sec4}
	Being motivated by the generalization of minimal excludant given in \cite{Wagner}, we generalize the concept of minimal excludants to the least $r$-gaps in $k$-colored partitions and obtain the generating functions for $\sigma{gap_k^r}(\nu),$ $\sigma{ogap_k^r}(\nu),$ and $\sigma{egap_k^r}(\nu)$. These functions complement the results presented in Theorems $\ref{Mex_k},$ $\ref{Moex_k},$ and $\ref{Meex_k},$ respectively.
	By observation, one can see that for $r=1,$ the least $r$-gap, least odd $r$-gap, and least even $r$-gap of a $k$-colored partition $\Lambda$ coincide with the minimal excludant, minimal odd excludant, and minimal even excludant of a $k$-colored partition $\Lambda,$ respectively.
	\subsection{Generating functions for $\sigma{gap_k^r}(\nu),$ $\sigma{ogap_k^r}(\nu),$ and $\sigma{egap_k^r}(\nu)$ }
	\begin{theorem}\label{least_gap}
		For fixed $r,k\in\mathbb{N}$ and $\nu \in \mathbb{Z}_{\geq0},$ we have 
		$$ \sum\limits_{\nu=0}^{\infty}\sigma{gap_k^r}(\nu)q^{\nu} = \dfrac{(q^{2kr};q^{2kr})_\infty(-q^{kr};q^{kr})_\infty}{(q;q)_{\infty}^k}.	$$
	\end{theorem}
	\begin{theorem}\label{least_odd_gap}
		For fixed $r,k\in\mathbb{N}$ and $\nu \in \mathbb{Z}_{\geq0},$ we have 
		$$ \sum\limits_{\nu=0}^{\infty}\sigma{ogap_k^r}(\nu)q^{\nu} = \dfrac{(-q^{kr};q^{2kr})_{\infty}^2(q^{2kr};q^{2kr})_\infty}{(q;q)_{\infty}^k}.	$$
	\end{theorem}
	\begin{theorem}\label{least_even_gap}
		For fixed $r,k\in\mathbb{N}$ and $\nu \in \mathbb{Z}_{\geq0},$ we have 
		$$ \sum\limits_{\nu=0}^{\infty}\sigma{egap_k^r}(\nu)q^{\nu} = \dfrac{2(q^{4kr};q^{4kr})_\infty(-q^{2kr};q^{2kr})_\infty}{(q;q)_{\infty}^k}.	$$
	\end{theorem}
	\begin{remark}
		For $k=1,$ the result of Theorem  $\ref{least_gap}$ coincides with the equation $\eqref{least_gap_merca}$.
	\end{remark}
	\subsection{Asymptotic behavior of $\sigma{gap_k^r}(\nu),$ $\sigma{ogap_k^r}(\nu),$ and $\sigma{egap_k^r}(\nu)$ }
	\noindent Here, we study the asymptotic formulas for $\sigma{gap_k^r}(\nu),$ $\sigma{ogap_k^r}(\nu),$ and $\sigma{egap_k^r}(\nu)$.
	\begin{theorem}\label{Asymp_Least_gap}
		For fixed $r,k\in\mathbb{N},$ we have
		$$\sigma{gap_k^r(\nu)}\sim \frac{k^{\frac{k-2}{4}}}{2^{\frac{k+3}{2}}\cdot6^{\frac{k}{4}}\cdot r^{\frac{1}{2}}\cdot \nu^{\frac{k+2}{4}}}\exp\!\left( 2\pi \sqrt{\frac{k\nu}{6}} \right)~~~~\text{as}~\nu \to \infty.$$
	\end{theorem}
	\begin{theorem}\label{Asymp_Least_O_E_gap}
		For fixed $r,k\in\mathbb{N},$ we have
		$$\sigma{ogap_k^r(\nu)}\sim\sigma{egap_k^r(\nu)}\sim \frac{k^{\frac{k-2}{4}}}{2^{\frac{k+2}{2}}\cdot6^{\frac{k}{4}}
			\cdot r^{\frac{1}{2}}\cdot \nu^{\frac{k+2}{4}}}\exp\!\left( 2\pi \sqrt{\frac{k\nu}{6}} \right)~~~~\text{as}~\nu \to \infty.$$
	\end{theorem}
	\noindent Furthermore, building upon the work of Baruah \textit{et al.}
	\cite{baruah}, we generalize the functions $\sigma_o mex_k(\nu)$ and $\sigma_e mex_k(\nu)$ to the case of least $r$-gaps and derive corresponding generating functions for them. This work complements the  results presented in Theorems \ref{baruah_odd} and \ref{baruah_even}.
	\begin{theorem} \label{least_gap_odd}
		For fixed $r,k\in\mathbb{N}$ and $\nu \in \mathbb{Z}_{\geq0},$ we have
		$$\sum\limits_{\nu=0}^{\infty}\sigma_o gap_k^r(\nu)q^{\nu} =\dfrac{1}{2(q;q)^k_\infty}\left\{(-q^{kr};q^{kr})_\infty(q^{2kr};q^{2kr})_\infty+(q^{kr};q^{kr})^3_\infty\right\}.$$
	\end{theorem}
	\begin{theorem} \label{least_gap_even}
		For fixed $r,k\in\mathbb{N}$ and $n \in \mathbb{Z}_{\geq0},$ we have
		$$\sum\limits_{\nu=0}^{\infty}\sigma_e gap_k^r(\nu)q^{\nu} =\dfrac{1}{2(q;q)^k_\infty}\left\{(-q^{kr};q^{kr})_\infty(q^{2kr};q^{2kr})_\infty-(q^{kr};q^{kr})^3_\infty\right\}.$$
	\end{theorem}
	
	\subsection{Proofs of Theorems \ref{least_gap}~-  \ref{least_even_gap}, \ref{least_gap_odd}, and \ref{least_gap_even}}
	\begin{proof}[\normalfont{\textbf{Sketch of the Proof of Theorem \ref{least_gap}}}]
		We fix $r,~k\in\mathbb{N}$ and define $\mathcal{M}_k^r(z_1,\dots,z_k,q)$ as the series in which the coefficient of $z_\mathfrak{j}^\mathfrak{m}q^{\nu}$, denoted by $|\mathcal{P}_k^r(\mathfrak{m}_\mathfrak{j},\mathfrak{i},\nu)|,$ represents the count of $k$-colored partitions $\Lambda$ of a positive integer $\nu$ with $gap_k^r(\Lambda)=\mathfrak{m}_\mathfrak{j},$ where $1\leq \mathfrak{j}\leq k.$
		\begin{align*}
			&\mathcal{M}_k^r(z_1,z_2,\dots,z_k,q):=\sum\limits_{\nu=0}^{\infty}\sum\limits_{\mathfrak{m}=1}^{\infty}\sum\limits_{\mathfrak{j}=1}^{k}\sum\limits_{\mathfrak{i}=0}^{r-1}|\mathcal{P}_k^r(\mathfrak{m}_\mathfrak{j},\mathfrak{i},\nu)|z_\mathfrak{j}^\mathfrak{m}q^{\nu}\\
			&=\sum\limits_{\mathfrak{m}=1}^{\infty}\sum\limits_{\mathfrak{j}=1}^{k}\sum\limits_{\mathfrak{i}=0}^{r-1}q^{\scalebox{0.55}{$
					\underbrace{\underbrace{(1+\cdots+1)}_{r\text{ times}} + \cdots + 
						\underbrace{(1+\cdots+1)}_{r\text{ times}}}_{k\text{ times}} +
					\underbrace{\underbrace{(2+\cdots+2)}_{r\text{ times}} + \cdots +
						\underbrace{(2+\cdots+2)}_{r\text{ times}}}_{k\text{ times}}+\cdots+\underbrace{\underbrace{((\mathfrak{m}-1)+\cdots+(\mathfrak{m}-1))}_{r\ \text{times}} + \cdots + 
						\underbrace{((\mathfrak{m}-1)+\cdots+(\mathfrak{m}-1))}_{r\ \text{times}}}_{k\ \text{times}}
					$}}\\
			&\quad\quad\quad\quad\quad\quad~ q^{\scalebox{0.55}{$
					\underbrace{\underbrace{(\mathfrak{m}+\cdots+\mathfrak{m})}_{r\ \text{times}} + \cdots + 
						\underbrace{(\mathfrak{m}+\cdots+\mathfrak{m})}_{r\ \text{times}}}_{(\mathfrak{j}-1)\ \text{times}}+\underbrace{(\mathfrak{m}+\cdots+\mathfrak{m})}_{\mathfrak{i} \text{ times}}$}}(1-q^\mathfrak{m})z_\mathfrak{j}^\mathfrak{m}\left(\prod_{\nu=1}^{\infty}\frac{1}{(1-q^{\nu})^k}\right)\\
			&=\sum\limits_{\mathfrak{m}=1}^{\infty}\sum\limits_{\mathfrak{j}=1}^{k}\sum\limits_{\mathfrak{i}=0}^{r-1}	q^{k\cdot r(1+2+\cdots+(\mathfrak{m}-1))+(\mathfrak{j}-1)r\cdot \mathfrak{m}+\mathfrak{i}\cdot \mathfrak{m}}(1-q^\mathfrak{m})z_\mathfrak{j}^\mathfrak{m}\left(\prod_{\nu=1}^{\infty}\frac{1}{(1-q^{\nu})^k}\right).
		\end{align*}
		Now, by following the method used in the proof of Theorem \ref{Mex_k}, we obtain
		$$ \sum\limits_{\nu=0}^{\infty}\sigma{gap_k^r}(\nu)q^{\nu} = \dfrac{(q^{2kr};q^{2kr})_\infty(-q^{kr};q^{kr})_\infty}{(q;q)_{\infty}^k}.	$$
		This concludes the proof.
	\end{proof}
	\noindent By employing a similar approach as in the sketch of the proof of Theorem \ref{Mex_k}, one can prove Theorems \ref{least_odd_gap}, \ref{least_even_gap}, \ref{least_gap_odd}, and \ref{least_gap_even}.
	\subsection{Proofs of Theorems \ref{Asymp_Least_gap} and \ref{Asymp_Least_O_E_gap}}
	Following the line of proof of Lemma \ref{lemma_1}, it can be shown that the sequences $\{\sigma gap_k^r(\nu)\},$ $\{\sigma ogap_k^r(\nu)\}$, and $\{\sigma egap_k^r(\nu)\}$ are weakly increasing. The proofs of Theorems \ref{Asymp_Least_gap} and \ref{Asymp_Least_O_E_gap} then follow on the similar arguments as in the proofs of Theorems \ref{Asymp_mex_k} and \ref{Asymp_moex_k_meex_k} respectively.
	\section{Minimal excludant in generalized cubic partitions}\label{sec5}
	A cubic partition of a positive integer $\nu$ is a partition in which each even part may appear in two different colors. Let $C(\nu)$ denote the set of all cubic partitions of a positive integer $\nu$. For example, 
	$$C(4)=\{4_2,~4_1,~3_1+1_1,~2_2+2_2,~2_2+2_1,~2_1+2_1,~2_2+1_1+1_1,~2_1+1_1+1_1,
	~1_1+1_1+1_1+1_1\},$$
	and $|C(4)| = 9.$ Now, let the minimal excludant of a cubic partition $\Lambda$ be denoted by  $mex_{\{2,1\}}(\Lambda).$ For example, the minimal excludants in the cubic partitions of $4$ are $1_1,~1_1,~2_1,~1_1,~1_1,~1_1,~2_1,~2_2,$ and $2_1,$ respectively. The sum of the minimal excludants across all cubic partitions of a positive integer $\nu$ is denoted by $\sigma mex_{\{2,1\}}(\nu)$. Thus, $\sigma mex_{\{2,1\}}(4)=13$.\\
	We define a generalized form of cubic partitions wherein each even part is allowed to appear in $l$ distinct colors and each odd part may appear in $s$ distinct colors. When $s=l,$ this reduces to the case of $l$-colored partitions. This generalization is referred to as the $\{l,s\}$-colored
	partition. Let $|C_{\{l,s\}}(\nu)|$ represent the count of the $\{l,s\}$-colored partitions of $\nu$, then the generating function for $|C_{\{l,s\}}(\nu)|$ is
	$$\sum\limits_{\nu=0}^{\infty}|C_{\{l,s\}}(\nu)|q^{\nu} =\prod _{\nu=1}^{\infty}\frac{1}{(1-q^{2\nu-1})^s~(1-q^{2\nu})^l}=\frac{1}{(q;q^2)_\infty^s~(q^2;q^2)_\infty^l} .$$
	For fixed $l,~s\in\mathbb{N},$ let $\sigma{mex}_{\{l,s\}}(\nu)$ represents the sum of the minimal excludants across all the $\{l,s\}$-colored partitions of $\nu$.
	\begin{theorem} \label{Mex_cubic}
		For fixed $l,s\in\mathbb{N}$ and $\nu \in \mathbb{Z}_{\geq0},$ we have 
		\begin{align*}
			\sum\limits_{\nu=0}^{\infty}\sigma{mex}_{\{l,s\}}(\nu)q^{\nu}=\frac{(-q^{s+2l};q^{2(s+l)})_\infty(-q^s;q^{2(s+l)})_\infty(q^{2(s+l)};q^{2(s+l)})_\infty}{(q;q^2)_\infty^s~(q^2;q^2)_\infty^l}.
		\end{align*}
	\end{theorem}
	\begin{proof}[\normalfont\textbf{Proof}]
		We fix $l,s\in\mathbb{N}$ and define the series $\mathcal{M}_{\{l,s\}}(z_1,\dots,z_l,w_1,\dots,w_s,q)$ which is divided into two subseries such that 
		\begin{eqnarray}\label{mex_l_s}
			\mathcal{M}_{\{l,s\}}(z_1,\dots,z_l,w_1,\dots,w_s,q) = \mathcal{M}^E_{\{l,s\}}(z_1,\dots,z_l,q) + \mathcal{M}^O_{\{l,s\}}(w_1,\dots,w_s,q)
		\end{eqnarray}
		In the series $\mathcal{M}^E_{\{l,s\}}(z_1,\dots,z_l,q),$ the coefficient of $z_\mathfrak{i}^{2g}q^{\nu},$ denoted by $|\mathcal{P}_{\{l,s\}}((2g)_\mathfrak{i},\nu)|,$ represents the count of $\{l,s\}$-colored partitions of $\nu$ with minimal excludant = $(2g)_\mathfrak{i}$ for some $g\in\mathbb{N},$ where $1\leq \mathfrak{i}\leq l$.
		\begin{eqnarray}
			&&\mathcal{M}^E_{\{l,s\}}(z_1,\dots,z_l,q)  := \sum\limits_{\nu=0}^{\infty}\sum\limits_{g=1}^{\infty}\sum\limits_{\mathfrak{i}=1}^{l}|\mathcal{P}_{\{l,s\}}((2g)_\mathfrak{i},\nu)|z_\mathfrak{i}^{2g}q^{\nu}\nonumber\\
			&&=\left(\frac{\sum\limits_{g=1}^{\infty}\sum\limits_{\mathfrak{i}=1}^{l}q^{\{\scalebox{0.6}{$\underbrace{(1+\cdots+1)}_{s~\text{times}}+\underbrace{(2+\cdots+2)}_{l~\text{times}}+\cdots+\underbrace{((2g-1)+\cdots+(2g-1))}_{s~\text{times}}+\underbrace{(2g+\cdots+2g)}_{(\mathfrak{i}-1)~\textup{times}}$}\}}\left(1-q^{2g}\right)z_\mathfrak{i}^{2g}}{(q;q^2)_\infty^s~(q^2;q^2)_\infty^l}\right)\nonumber\\
			&&=\left(\frac{\sum\limits_{g=1}^{\infty}\sum\limits_{\mathfrak{i}=1}^{l}q^{\{s\cdot1+l\cdot2+s\cdot3+l\cdot4+\cdots+s\cdot(2g-1)+(\mathfrak{i}-1)2g\}}\left(1-q^{2g}\right)z_\mathfrak{i}^{2g}}{(q;q^2)_\infty^s~(q^2;q^2)_\infty^l}\right)\nonumber\\
			&&=\left(\frac{\sum\limits_{g=1}^{\infty}\sum\limits_{\mathfrak{i}=1}^{l}q^{\{s(1+3+\cdots+(2g-1))+l(2+4+\cdots+(2g-2))+(\mathfrak{i}-1)2g\}}\left(1-q^{2g}\right)z_\mathfrak{i}^{2g}}{(q;q^2)_\infty^s~(q^2;q^2)_\infty^l}\right)\nonumber\\
			&&=\sum\limits_{g=1}^{\infty}\sum\limits_{\mathfrak{i}=1}^{l}q^{\{sg^2+lg(g-1)+(\mathfrak{i}-1)2g\}}\left(1-q^{2g}\right)z_\mathfrak{i}^{2g}\left(\frac{1}{(q;q^2)_\infty^s~(q^2;q^2)_\infty^l}\right).\label{mex_l,s_even}
		\end{eqnarray}
		Similarly, in the series $\mathcal{M}^O_{\{l,s\}}(w_1,\dots,w_s,q),$ the coefficient of $w_\mathfrak{j}^{(2g-1)}q^{\nu},$ denoted by $|\mathcal{P}_{\{l,s\}}((2g-1)_\mathfrak{j},\nu)|,$ represents the count of $\{l,s\}$-colored partitions of $\nu$ with minimal excludant = $(2g-1)_\mathfrak{j}$ for some $g\in\mathbb{N},$ where $1\leq \mathfrak{j}\leq s$.
		\begin{eqnarray}
			&&\mathcal{M}^O_{\{l,s\}}(w_1,\dots,w_s,q)  := \sum\limits_{\nu=0}^{\infty}\sum\limits_{g=1}^{\infty}\sum\limits_{\mathfrak{j}=1}^{s}|\mathcal{P}_{\{l,s\}}((2g-1)_\mathfrak{j},\nu)|w_\mathfrak{j}^{2g-1}q^{\nu}\nonumber\\
			&&=\left(\frac{\sum\limits_{g=1}^{\infty}\sum\limits_{\mathfrak{j}=1}^{s}q^{\{\scalebox{0.6}{$\underbrace{(1+\cdots+1)}_{s~\text{times}}+\underbrace{(2+\cdots+2)}_{l~\text{times}}+\cdots+\underbrace{((2g-2)+\cdots+(2g-2))}_{l~\text{times}}+\underbrace{((2g-1)+\cdots+(2g-1))}_{(\mathfrak{j}-1)~\textup{times}}$}\}}\left(1-q^{2g-1}\right)w_\mathfrak{j}^{2g-1}}{(q;q^2)_\infty^s~(q^2;q^2)_\infty^l}\right)\nonumber\\
			&&=\left(\frac{\sum\limits_{g=1}^{\infty}\sum\limits_{\mathfrak{j}=1}^{s}q^{\{s\cdot1+l\cdot2+s\cdot3+l\cdot4+\cdots+l\cdot(2g-2)+(\mathfrak{j}-1)(2g-1)\}}\left(1-q^{2g-1}\right)w_\mathfrak{j}^{2g-1}}{(q;q^2)_\infty^s~(q^2;q^2)_\infty^l}\right) \nonumber\\
			&&=\left(\frac{\sum\limits_{g=1}^{\infty}\sum\limits_{\mathfrak{j}=1}^{s}q^{\{s(1+3+\cdots+(2g-3))+l(2+4+\cdots+(2g-2))+(\mathfrak{j}-1)(2g-1)\}}\left(1-q^{2g-1}\right)w_\mathfrak{j}^{2g-1}}{(q;q^2)_\infty^s~(q^2;q^2)_\infty^l}\right)\nonumber\\
			&&=\sum\limits_{g=1}^{\infty}\sum\limits_{\mathfrak{j}=1}^{s}q^{\{s(g-1)^2+lg(g-1)+(\mathfrak{j}-1)(2g-1)\}}\left(1-q^{2g-1}\right)w_\mathfrak{j}^{2g-1}\left(\frac{1}{(q;q^2)_\infty^s~(q^2;q^2)_\infty^l}\right).\label{mex_l,s_odd}
		\end{eqnarray}
		Now, by using equations \eqref{mex_l_s}, \eqref{mex_l,s_even}, and \eqref{mex_l,s_odd}, we obtain
		\begin{align*}
			&\mathcal{M}_{\{l,s\}}(z_1,\dots,z_l,w_1,\dots,w_s,q) 
			=\sum\limits_{g=1}^{\infty}\left\{\left(\sum\limits_{\mathfrak{i}=1}^{l}\underbrace{q^{\left\{sg^2+lg(g-1)+(\mathfrak{i}-1)2g\right\}}\left(1-q^{2g}\right)z_\mathfrak{i}^{2g}}_{\textup{Part I}}\right)\right.\\
			&\quad\quad\quad+\left.\left(\sum\limits_{\mathfrak{j}=1}^{s}\underbrace{q^{\left\{s(g-1)^2+lg(g-1)+(\mathfrak{j}-1)(2g-1)\right\}}\left(1-q^{2g-1}\right)w_\mathfrak{j}^{2g-1}}_{\textup{Part II}}\right)\right\}\left(\frac{1}{(q;q^2)_\infty^s~(q^2;q^2)_\infty^l}\right).
		\end{align*}
		Now,
		\begin{eqnarray}
			&&\sum\limits_{\nu=0}^{\infty}\sigma{mex}_{\{l,s\}}(\nu)q^{\nu}\nonumber\\
			&&:=
			\sum\limits_{g=1}^{\infty}\left\{\left(\sum\limits_{\mathfrak{i}=1}^{l}\left(\left.\dfrac{\partial}{\partial z_\mathfrak{i}}\right|_{z_\mathfrak{i}=1}\textup{Part I}\right)\right)+\right.\left.\left(\sum\limits_{\mathfrak{j}=1}^{s}\left(\left.\dfrac{\partial}{\partial w_\mathfrak{j}}\right|_{w_\mathfrak{j}=1}\textup{Part II}\right)\right)\right\}\left(\frac{1}{(q;q^2)_\infty^s~(q^2;q^2)_\infty^l}\right)\nonumber\\
			&&=\sum\limits_{g=1}^{\infty}\left\{\left(\sum\limits_{\mathfrak{i}=1}^{l}2gq^{\left\{sg^2+lg(g-1)+(\mathfrak{i}-1)2g\right\}}\left(1-q^{2g}\right)\right)+\right.\nonumber\\
			&&~~~~~~~\left.\left(\sum\limits_{\mathfrak{j}=1}^{s}(2g-1)q^{\left\{s(g-1)^2+lg(g-1)+(\mathfrak{j}-1)(2g-1)\right\}}\left(1-q^{2g-1}\right)\right)\right\}\left(\frac{1}{(q;q^2)_\infty^s~(q^2;q^2)_\infty^l}\right)\nonumber\\
			&&= \left(\frac{\sum\limits_{g=1}^{\infty}2gq^{sg^2+lg(g-1)}\left(1-q^{2lg}\right)+(2g-1)q^{s(g-1)^2+lg(g-1)}\left(1-q^{s(2g-1)}\right)}{(q;q^2)_\infty^s~(q^2;q^2)_\infty^l}\right)\nonumber \\
			&&=\! \left(\!\frac{\!\sum\limits_{g=1}^{\infty}\!2gq^{sg^2+lg(g-1)}-2gq^{sg^2+lg(g+1)}+(2g-1)q^{s(g-1)^2+lg(g-1)}-(2g-1)q^{sg^2+lg(g-1)}}{(q;q^2)_\infty^s~(q^2;q^2)_\infty^l}\!\!\right)\nonumber \\
			&&=\left(\frac{\sum\limits_{g=1}^{\infty}(2g-2g+1)q^{sg^2+lg(g-1)}-2gq^{sg^2+lg(g+1)}+(2g-1)q^{s(g-1)^2+lg(g-1)}}{(q;q^2)_\infty^s~(q^2;q^2)_\infty^l}\right)\nonumber\\
			&&=\left(\frac{\sum\limits_{g=1}^{\infty}q^{sg^2+lg(g-1)}-\sum\limits_{g=1}^{\infty}2gq^{sg^2+lg(g+1)}+\sum\limits_{g=0}^{\infty}(2g+1)q^{sg^2+lg(g+1)}}{(q;q^2)_\infty^s~(q^2;q^2)_\infty^l}\right)\nonumber\\
			&&=\left(\sum\limits_{g=1}^{\infty}q^{sg^2+lg(g-1)}+1+\sum\limits_{g=1}^{\infty}q^{sg^2+lg(g+1)}\right)\left(\frac{1}{(q;q^2)_\infty^s~(q^2;q^2)_\infty^l}\right)\nonumber\\
			&&=\left(\sum\limits_{g=-\infty}^{\infty}q^{sg^2+lg(g+1)}\right)\left(\frac{1}{(q;q^2)_\infty^s~(q^2;q^2)_\infty^l}\right).\label{mex_sum_in_cubic}
		\end{eqnarray}
		Using \eqref{JTP_3} and \eqref{mex_sum_in_cubic}, we obtain
		\begin{align*}
			\sum\limits_{\nu=0}^{\infty}\sigma{mex}_{\{l,s\}}(\nu)q^{\nu}=\frac{(-q^{s+2l};q^{2(s+l)})_\infty(q^{2(s+l)};q^{2(s+l)})_\infty(-q^s;q^{2(s+l)})_\infty}{(q;q^2)_\infty^s~(q^2;q^2)_\infty^l}.
		\end{align*}
		This concludes the proof.
	\end{proof}
	\begin{corollary}\label{cor1}
		For a positive integer $\nu,$ we have 
		\begin{align*}
			\sum\limits_{\nu=0}^{\infty}\sigma{mex}_{\{2,1\}}(\nu)q^{\nu}=\frac{(-q^{5};q^{6})_\infty(q^{6};q^{6})_\infty(-q;q^{6})_\infty}{(q;q^2)_\infty(q^2;q^2)_\infty^2}.
		\end{align*}
	\end{corollary}
	\begin{proof}[\normalfont\textbf{Proof}] 
		The argument proceeds in the same manner as the proof of Theorem \ref{Mex_cubic}, therefore, we omit the details..
	\end{proof}
	\begin{theorem} \label{Asymp_cubic}
		We have
		$$\sigma{mex}_{\{2,1\}}(\nu)\sim \frac{1}{4\sqrt{6}\nu}\exp\left(\pi\sqrt{\nu}\right)~~~~\text{as}~\nu \to \infty.$$
		
	\end{theorem}
	\begin{proof}[\normalfont\textbf{Proof}]
		First we note that the sequence $\{\sigma{mex}_{\{2,1\}}(\nu)\}$ is weakly increasing due to the same reasoning as given in the Lemma \ref{lemma_1}. Now, let
		\begin{align*}
			\mathcal{D}(q):=\sum\limits_{\nu=0}^{\infty}\sigma{mex}_{\{2,1\}}(\nu)q^{\nu}
			&=\frac{(-q^{5};q^{6})_\infty(q^{6};q^{6})_\infty(-q;q^{6})_\infty}{(q;q^2)_\infty(q^2;q^2)_\infty^2}\\
			&=\frac{(q^{6};q^{6})_\infty(-q;q^{2})_\infty}{(-q^3;q^6)_\infty(q;q)_\infty(q^2;q^2)_\infty}. 
		\end{align*}
		Multiplying and dividing by $(-q^2;q^2)_\infty(-q^6;q^6)_\infty,$ we get
		\begin{eqnarray}
			\mathcal{D}(q)&=\frac{(q^{6};q^{6})_\infty(-q;q)_\infty(-q^{6};q^{6})_\infty}{(-q^3;q^3)_\infty(q;q)_\infty(q^2;q^2)_\infty(-q^{2};q^{2})_\infty}.\label{D(q)}
		\end{eqnarray}
		Upon replacing $q$ with $e^{-y}$ in the equation \eqref{D(q)} and using equations \eqref{result_1} and \eqref{result_2}, we obtain
		\begin{align*}
			\mathcal{D}(e^{-y})&=\frac{(e^{-6y};e^{-6y})_\infty(-e^{-y};e^{-y})_\infty(-e^{-6y};e^{-6y})_\infty}{(-e^{-3y};e^{-3y})_\infty(e^{-y};e^{-y})_\infty(e^{-2y};e^{-2y})_\infty(-e^{-2y};e^{-2y})_\infty}\\
			&\sim\frac{\frac{1}{\sqrt2}e^{\frac{\pi^2}{12y}}\cdot\frac{1}{\sqrt2}e^{\frac{\pi^2}{72y}}\cdot\sqrt{\frac{y}{2\pi}}e^{\frac{\pi^2}{6y}}\cdot\sqrt{\frac{2y}{2\pi}}e^{\frac{\pi^2}{12y}}}{\frac{1}{\sqrt2}e^{\frac{\pi^2}{36y}}\cdot\sqrt{\frac{6y}{2\pi}}e^{\frac{\pi^2}{36y}}\cdot\frac{1}{\sqrt2}e^{\frac{\pi^2}{24y}}}
			=\sqrt{\frac{y}{6\pi}}e^{\frac{\pi^2}{4y}} \text{~~~as~} y\rightarrow0^{+}.
		\end{align*}
		Using Theorem \ref{ingham} with $\rho=\frac{1}{\sqrt{6\pi}},$ $\tau=\frac{1}{2},$ and $\xi=\frac{\pi^2}{4}$, we get 
		\begin{align*}
			\sigma{mex}_{\{2,1\}}(\nu)\sim \frac{\frac{1}{\sqrt{6\pi}}\sqrt{\frac{\pi^2}{4}}}{2\sqrt{\pi}\nu}\exp\left(2\sqrt{\frac{\pi^2 \nu}{4}}\right)
			=\frac{1}{4\sqrt{6}\nu}\exp\left(\pi\sqrt{\nu}\right).
		\end{align*}
		This concludes the proof.
	\end{proof}
	\section{Concluding remarks}\label{sec6}
	This article presents a comprehensive study of the minimal excludants in the colored partitions. We have developed numerous generating functions that are linked to minimal excludants, with a particular emphasis on cubic partitions and $k$-colored partitions. We also examined the asymptotic formulas of the corresponding generating functions and Mathematica was used to verify the validity of our findings. In addition, we introduce a natural generalization of the minimal excludant, namely the least $r$-gap in colored partitions.\\
	An interesting aspect of this work is that it aligns with and extends previous results on minimal excludants by generalizing them from ordinary partitions to $k$-colored partitions. \\
	Several promising avenues for future research remain open within the broader context of minimal excludants. One direction is to investigate the arithmetic properties of the sums of the minimal excludants and least $r$-gaps across all colored partitions of a positive integer $\nu.$ Moreover, developing combinatorial proofs for the presented generating functions may yield additional insights.  \\
	Another direction is to extend this study to $n$-color partitions. During this work, we derived the generating function that captures the sum of the minimal excludants across all $n$-color partitions of a positive integer $\nu.$ This generating function is given by
	$$ \sum\limits_{\nu=0}^{\infty}\sigma{mex_{n}}(\nu)q^\nu =\sum\limits_{\mathfrak{m}=1}^{\infty}\mathfrak{m}q^{\left\{\frac{\mathfrak{m}(\mathfrak{m}-1)(2\mathfrak{m}-1)}{6}\right\}}(1-q^{\mathfrak{m}^2})\prod\limits_{\nu=1}^{\infty}\dfrac{1}{(1-q^\nu)^\nu},	$$
	where $\sigma{mex_{n}}(\nu)$ represents the sum of the minimal excludants across all $n$-colored partitions of a positive integer $\nu$. \\
	
	\section*{Acknowledgements}
	The first, second, and fourth author would like to thank Department of Mathematics, Thapar Institute of Engineering and Technology, Patiala for providing excellent working conditions. The second author would like to acknowledge the support of the National Board for Higher Mathematics (NBHM) via Project Reference No. 02011/40/2025/NBHM (R.P)/R\&D II/16129.
	\section*{Data Availability Statement}
The manuscript has no associated data.
	
\label{'ubl'}  
\end{document}